\documentclass[oneside,english]{amsart}
\usepackage[T1]{fontenc}
\usepackage[latin9]{inputenc}
\usepackage{amstext}
\usepackage{amsthm}
\usepackage{amssymb}

\makeatletter
\numberwithin{equation}{section}
\numberwithin{figure}{section}
\theoremstyle{plain}
\newtheorem{thm}{\protect\theoremname}
\theoremstyle{plain}
\newtheorem{conjecture}[thm]{\protect\conjecturename}
\theoremstyle{plain}
\newtheorem{lem}[thm]{\protect\lemmaname}
\theoremstyle{remark}
\newtheorem{rem}[thm]{\protect\remarkname}
\theoremstyle{plain}
\newtheorem{cor}[thm]{\protect\corollaryname}

\usepackage{fullpage}

\makeatother

\usepackage{babel}
\providecommand{\conjecturename}{Conjecture}
\providecommand{\corollaryname}{Corollary}
\providecommand{\lemmaname}{Lemma}
\providecommand{\remarkname}{Remark}
\providecommand{\theoremname}{Theorem}

\begin{document}

\title{A Inequality for Non-Microstates Free Entropy Dimension for Crossed
Products by Finite Abelian Groups}

\author{Dimitri Shlyakhtenko}
\begin{abstract}
For certain generating sets of the subfactor pair $M\subset M\rtimes G$
where $G$ is a finite abelian group we prove an approximate inequality
between their non-microstates free entropy dimension, resembling the
Shreier formula for ranks of finite index subgroups of free groups.
As an application, we give bounds on free entropy dimension of generating
sets of crossed products of the form $M\rtimes(\mathbb{Z}/2\mathbb{Z})^{\oplus\infty}$
for a large class of algebras $M$. 
\end{abstract}

\address{UCLA Department of Mathematics, UCLA, Los Angeles, CA 90095}

\email{shlyakht@math.ucla.edu}

\dedicatory{To the memory of V.F.R. Jones. }
\maketitle

\section{Introduction.}

The famous Jones index \cite{jones:index} is the von Neumann algebra
extension of the group-subgroup index and is defined for any inclusion
$M_{0}\subset M_{1}$ of II$_{1}$ factors. An open question in von
Neumann algebra theory is to find an analog of the Shreier\textquoteright s
formula for the number of generators of finite-index subgroups of
non-abelian free groups $\mathbb{F}_{n}$. For example, one expects
that the ``number of generators'' of a finite-index subfactor $M_{0}\subset M_{1}=L(\mathbb{F}_{n})$
should be $1+(n-1)[M:N]$. Indeed, specific subfactors of $L(\mathbb{F}_{r})$
constructed via amalgamated free producs \cite{popa:markov,radulescu:subfact,shlyakht-popa:universal,guionnet-jones-shlyakht:freePlanar,guionnet-shlyakht-jones:semifinite}
have sets of generators for which such a formula holds. However, there
is little that is known in general, even for index $2$.

Returning to group theory, let $H\subset G$ be a finite-index inclusion
of groups. Denoting by $\beta_{j}^{2}(G)$ the $L^{2}$-Betti numbers
of $G$ one has the following generalization of Shreier's formula
(see e.g. \cite{luck:book}):
\[
\beta_{j}^{(2)}(G)=[G:H]^{-1}\beta_{j}^{(2)}(H).
\]
(Schreier's formula corresponds to the case $j=1$ and involves the
equality $\beta_{1}^{(2)}(\mathbb{F}_{r})=r-1$). A similar formula
is true for finite-index inclusions of tracial algebras \cite{thom:L2index}. 

Voiculescu's free entropy dimension $\delta_{0}$ \cite{dvv:entropy2,dvv:entropy3}
takes the value $r$ on a set of generators of $L(\mathbb{F}_{r})$;
more generally, its value is related to $L^{2}$ Betti numbers \cite{connes-shlyakht:l2betti,shlyakht-mineyev:freedim,shlyakht:sofic}.
Therefore one expects a statement of the following kind: given a finite
index inclusion $M_{0}\subset M_{1}$, for any generating set $S_{0}$
for $M_{0}$ there exists a generating set for $M_{1}$ (and conversely,
for any generating set $S_{1}$ for $M_{1}$ there exists a generating
set $S_{0}$ for $M_{0}$) so that
\begin{equation}
\delta_{0}(S_{0})-1=[M_{1}:M_{0}]^{-1}(\delta_{0}(S_{1})-1).\label{eq:SchreierFormulaDelta0}
\end{equation}
However, at present only inequalities of the form $\delta_{0}(S_{1})\leq\delta_{0}(S_{0})$
are available \cite{jung:subfactorDimension}. 

In this paper we show that for very specific examples of subfactors,
namely subfactors of the form $M_{0}\subset M_{1}=M_{0}\rtimes G$
with $G$ a finite abelian group, we can find generators $M_{0}$
and $M_{1}$ for which the non-microstates free entropy dimension
\cite{dvv:entropy5} analog of (\ref{eq:SchreierFormulaDelta0}) holds
with an arbitrary small error. More precisely, we prove that given
$\epsilon>0$ there exist generating sets $S_{0}$ for $M_{0}$ and
$S_{1}$ for $M_{1}$ for which
\[
\delta^{\star}(S_{0})-1=[M_{1}:M_{0}]^{-1}(\delta^{\star}(S_{1})-1)+\epsilon.
\]

Our result is interesting in connection with the following question.
Let $H$ be a finitely generated group, and let $\alpha$ be an action
of some infinite group $G$ on $H$. Then it is known \cite[Theorem 6.8]{gaboriau:ell2}
that $\beta_{1}^{(2)}(H\rtimes_{\alpha}G)=0$. In case that $G=\mathbb{Z}$
and $H$ is finitely presented sofic, this implies that the von Neumann
algebra $M_{1}=L(H)\rtimes\mathbb{Z}$ is strongly one-bounded \cite{jung:onebounded,hayes:oneBounded,shlyakht:sofic};
in particular for any generating $S$ set of $M_{1}$, $\delta_{0}(S)=1$.
This leads us to the following conjecture:
\begin{conjecture}
Let $M_{0}$ be a finitely-generated von Neumann algebra, and let
$G$ be an action of a infinite group on $M_{0}$. Then $M_{1}=M_{0}\rtimes_{\alpha}G$
is strongly $1$-bounded.
\end{conjecture}

If true, the conjecture has a somewhat surprising consequence: it
would imply non-isomorphism of free group factors. Indeed, let $G$
be any infinite group so that $L(G)$ is $R^{\omega}$-embeddable
(e.g. $G=\mathbb{Z}$ or $G$ amenable), and regard $\mathbb{F}_{\infty}$
as the infinite free product of copies of $\mathbb{Z}$ indexed by
$G$. Then $G$ acts on this index set by permutations and thus on
$\mathbb{F}_{\infty}$; call this action $\alpha$. One can easily
see that the resulting semi-direct product is $\mathbb{Z}*G$, corresponding
to the extension $e\to\mathbb{F}_{|G|}\to\mathbb{Z}*G\stackrel{e*id}{\to}G\to e$.
Thus
\[
M_{1}=M_{0}\rtimes_{\alpha}G=L(\mathbb{F}_{\infty}\rtimes_{\alpha}G)=L(\mathbb{Z}*G)\cong L(\mathbb{Z})*L(G).
\]
The latter von Neumann algebra is known to be not strongly $1$-bounded;
in fact it has a generating set whose free entropy dimension is strictly
above $1$. On the other hand, if we were to assume that $L(\mathbb{F}_{\infty})\cong L(\mathbb{F}_{2})$
(or even finitely-generated), our conjecture (for the specific group
$G$) would imply that $M_{1}$ is strongly $1$-bounded and thus
all of its generating sets have free entropy dimension $1$, which
would be a contradiction.

While we are unable to even come close to proving the conjecture,
we are able to show that if $M_{0}$ is finitely-generated and in
addition $M_{0}\cong M_{2\times2}(M_{0})$ and $G=(\mathbb{Z}/2\mathbb{Z})^{\oplus\infty}$
then for any $\epsilon>0$, $M_{1}$ has generating sets with free
entropy dimension bounded by $1+\epsilon$. The proof is reminiscent
of Gaboriau's proof \cite{gaboriau:cost} and it is this connection
that inspired us to study the behavior of free entropy dimension under
crossed products.

\subsection*{Acknowledgements. }

The author is grateful to B. Hayes for showing him the reference \cite{gaboriau:cost}
and for useful discussions. This research was supported by NSF grant
DMS-2054450.

\section{Estimates on non-microstates free entropy dimension.}

\subsection{Special generators for crossed product subfactors.\label{subsec:SpecialGenerators}}

Let $M$ be a II$_{1}$ factor and let $\alpha$ be a properly outer
action of a finite abelian group $G$ on $M$. Consider the inclusion
of factors
\[
M^{G}\subset M\subset M\rtimes_{\alpha}G
\]
where $M^{G}$ is the fixed point algebra for the action $\alpha$.
It follows from Takesaki duality that $M\rtimes_{\alpha}G\cong M_{|G|\times|G|}(M^{G})$
, and moreover that the inclusion $M^{G}\subset M$ is isomorphic
to an inclusion of the form $M^{G}\subset M^{G}\rtimes_{\alpha'}\hat{G}$,
where $\hat{G}$ is the group dual of $G$ and $\alpha'$ is a certain
action related to the dual action of $\hat{G}$ on $M\rtimes_{\alpha}G$. 

Assume now that $M$ is finitely generated; therefore also $M\rtimes_{\alpha}G$
is finitely generated. Since $M_{|G|\times|G|}(M^{G})\cong M\rtimes_{\alpha}G$
we also know that $M^{G}$ is finitely generated. Let $X=(X_{1},\dots,X_{d})$
be a set of generators for $M^{G}$. Denote by $\hat{u}_{g}\in M^{G}\rtimes\hat{G}$,
$g\in\hat{G}$, the unitaries implementing $\alpha'$. Using the isomorphism
$(M^{G}\subset M)\subset(M^{G}\subset M^{G}\rtimes_{\alpha'}\hat{G})$,
we may view these unitaries as elements of $M$. Then the set $X\cup(\hat{u}_{g})_{g\in\hat{G}}$
generates $M$. Furthermore, if we denote by $u_{g}\in M\rtimes_{\alpha}G$,
$g\in G$, the unitaries implementing $\alpha$, then $X\cup(\hat{u}_{g})_{g\in\hat{G}}\cup(u_{g})_{g\in G}$
form a generating set of $M\rtimes_{\alpha}G$. Note that we have
the following relations:
\begin{align*}
u_{g}X_{j}u_{g}^{*} & =X_{j},\qquad g\in G,\ j=1,\dots,d\\
u_{g}v_{h}u_{g}^{*} & =\langle g,h\rangle v_{h},\qquad g\in G,\ h\in\hat{G}
\end{align*}
where we use $\langle\cdot,\cdot\rangle$ to denote the pairing between
the elements of $G$ and its dual $\hat{G}$. 

\subsection{Estimates on non-microstates free entropy dimension $\delta^{\star}$.}

In this paper it will be convenient to work with a non-selfadjoint
version of free entropy dimension. Given a non-commutative random
variables $Y$ in a tracial von Neumann algebra $M$ and a sub-algebra
$B\subset M$ , let $A=*\operatorname{-alg}(Y,B)$ and let $C$ be
a circular element free from $A$; we normalize $C$ so that $\tau(C^{*}C)=2$.
Consider the derivation $\partial_{Y}:A\to\textrm{span}\{ACA+AC^{*}A\}$
determined by the Leibniz rule and by $\partial_{Y}(Y)=C$, $\partial_{Y}(Y^{*})=C^{*}$,
$\partial_{Y}(b)=0$ for all $b\in B$. The vector $J(Y:B)\in L^{2}(A,\tau)$,
if it exists, is called the conjugate variable to $Y$ and is uniquely
determined by
\[
\langle J(Y:B),P\rangle_{L^{2}(A)}=\langle C^{*},P\rangle_{L^{2}(A*W^{*}(C))},\qquad\forall P\in A.
\]
Given $Y_{i}:i\in I$ the free entropy dimension is then determined
by
\[
\delta^{\star}(Y_{i}:i\in I)=2|I|-\limsup_{t\to0}t\Vert J(Y_{i}:(Y_{j}:j\in I\setminus\{i\})\Vert_{2}^{2}.
\]
It is not hard to see that our definition is equivalent to the usual
definition of $\delta^{\star}$ for self-adjoint variables, in that
$\delta^{\star}(Y_{i}:i\in I)=\delta_{s.a.}^{\star}(\textrm{Re}Y_{i},\textrm{Im}Y_{i}:i\in I)$.
\begin{lem}
\label{lem:compenentsEqual}Suppose that $\alpha$ is an action of
a finite abelian group on a tracial von Neumann algebra $M$, and
suppose that $Y_{j}\in M$, $j\in I$ are (not necessarily self-adjoint)
generators of $M$. Let $C_{j}^{(g)}$, $g\in G$ be circular elements
$*$-free from $M$, and extend $\alpha$ to $W^{*}(Y_{i}:i\in I)*W^{*}(C_{i}^{(g)}:i\in I,g\in G)$
by setting $\alpha_{g}(C_{i}^{(g')})=C_{i}^{(gg')}$. Let finally
$Y_{j}^{t}=Y_{j}+\sqrt{t}C_{j}^{(e)}$, where $e\in G$ is the neutral
element. Let
\[
\xi_{j}^{t}=J(Y_{j}^{t}:(Y_{i}^{t}:i\in I\setminus\{j\}))
\]
be the free conjugate variables.

For each $h\in\hat{G}$, denote by $(\bar{C}_{j}^{(h)})^{*}$ the
projection of $(C_{j}^{(e)})^{*}$ on to the linear subspace of $\operatorname{span}\{\alpha_{g}((C_{j}^{(e)})^{*}):g\in G\}$
consisting of vectors $x$ satisfying $\alpha_{g}(x)=\langle g,h\rangle x$. 

Then
\begin{equation}
\xi_{j}^{t}=|G|^{1/2}t^{-1/2}E_{W^{*}(Y_{j}^{t}:j\in I)}((\bar{C}_{j}^{(h)})^{*}),\qquad\forall h\in\hat{G}\label{eq:expectationCh}
\end{equation}
and also
\[
\langle\xi_{j}^{t},(\bar{C}_{j}^{(h)})^{*}\rangle=\frac{1}{|G|^{1/2}}\langle\xi_{j}^{t},(C_{j}^{(e)})^{*}\rangle=\frac{t^{1/2}}{|G|^{1/2}}\Vert\xi_{j}^{t}\Vert_{2}^{2},\qquad\forall j\in I,\ h\in\hat{G.}
\]
In particular, if we denote by $\bar{\xi}_{j}^{t,(h)}$ the projection
of $\xi_{j}^{t}$ onto the subspace of $\operatorname{span}\{\alpha_{g}(\xi_{j}^{t}):g\in G\}$
consisting of vectors $x$ satisfying $\alpha_{g}(x)=\langle g,h\rangle x$,
then
\begin{align*}
\Vert\bar{\xi}_{j}^{t,(h)}\Vert_{2}^{2} & =\frac{1}{|G|}\Vert\xi_{j}^{t}\Vert^{2},\qquad j\in I,\ h\in\hat{G}\\
\langle\bar{\xi}_{j}^{t,(h)},(C_{j}^{(e)})^{*}\rangle & =\frac{t^{1/2}}{|G|^{1/2}}\Vert\xi_{j}^{t}\Vert_{2}^{2}.
\end{align*}
so that all of the orthogonal components $\bar{\xi}_{j}^{t,(h)}$
in the decomposition $\xi_{j}^{t}=\sum_{h\in\hat{G}}\bar{\xi}_{j}^{t,(h)}$
have the same length and the same inner product with $(C_{j}^{(e)})^{*}$.
(It is worth noting that $\bar{\xi}_{j}^{t,(h)}\notin W^{*}(Y_{j}^{t}:j\in I)$
since that algebra is not invariant under the action $\alpha$.)
\end{lem}

\begin{proof}
By \cite{shlyakht:amalg} we may assume that there exists a family
of free creation operators $\ell_{j}^{(g)}$, $\hat{\ell}_{j}^{(g)}$
satisfying, for all $g,g'\in G$, $j,j'\in I$, $y\in M$:
\begin{align*}
(\ell_{j}^{(g)})^{*}y\ell_{j'}^{(g')} & =(\hat{\ell}^{(g)})^{*}y\hat{\ell}_{j'}^{(g')}=\delta_{g=g'}\delta_{j=j'}\tau(y)\\
(\hat{\ell}_{j}^{(g)})^{*}y\ell_{j'}^{(g')} & =(\ell^{(g)})^{*}y\hat{\ell}_{j'}^{(g')}=0
\end{align*}
and so that $C_{j}^{(g)}=\ell_{j}^{(g)}+(\hat{\ell}_{j}^{(g)})^{*}$.
The action $\alpha$ can be extended by putting $\alpha_{g}(\ell_{j}^{(g')})=\ell_{j}^{(gg')}$
and $\alpha_{g}(\hat{\ell}_{j}^{(g')})=\hat{\ell}_{j}^{(g^{-1}g')}$.
Denote by $\bar{\ell}_{j}^{(h)}$ (resp., $\bar{\hat{\ell}}_{j}^{(h)}$)
the projection of $\ell_{j}^{(h)}$ (resp., $\hat{\ell}_{j}^{(h)}$)
onto the linear subspace of $\operatorname{span}\{\hat{\ell}_{j}^{(g)}:g\in G\}$
(resp., $\operatorname{span}\{\hat{\ell}_{j}^{(g)}:g\in G\}$) consisting
of vectors $x$ satisfying $\alpha_{g}(x)=\langle g,h\rangle x$;
in this way we get that $\bar{C}_{j}^{(h)}=\bar{\ell}_{j}^{(h)}+(\bar{\hat{\ell}}_{j}^{(h)})^{*}$.
We can now verify the following equations:
\begin{align*}
(\bar{\ell}_{j}^{(h)})^{*}y\bar{\ell}_{j'}^{(h')} & =(\hat{\bar{\ell}}^{(h)})^{*}y\hat{\bar{\ell}}_{j'}^{(h')}=|G|^{-1}\delta_{h=h'}\delta_{j=j'}\tau(y)\\
(\hat{\bar{\ell}}_{j}^{(h)})^{*}y\bar{\ell}_{j'}^{(h')} & =(\bar{\ell}^{(h)})^{*}y\hat{\bar{\ell}}_{j'}^{(h')}=0
\end{align*}
From this it follows that $(\bar{C}_{j}^{(h)}:j\in I,h\in\hat{G})$
is a family of free circular operators (of variance $|G|^{-1}$) which
are free from $M$. Since for each $j\in I$ , $C_{j}^{(e)}=\sum_{h\in\hat{G}}\bar{C}_{j}^{(h)}$,
it follows from \cite{dvv:entropy5} and the equalities $J(C_{j})=C_{j}^{*}$,
$J(\bar{C}_{j}^{(h)})=|G|^{-1/2}(\bar{C}_{j}^{(h)})^{*}$ that
\begin{align*}
\xi_{j}^{t} & =J(Y_{j}+\sqrt{t}C_{j}^{(e)}:(Y_{i}^{t}:i\in I\setminus\{j\}))\\
 & =J((Y_{j}+\sqrt{t}\sum_{h'\neq h}\bar{C}_{j}^{(h')})+\sqrt{t}\bar{C}_{j}^{(h)}:(Y_{i}^{t}:i\in I\setminus\{j\}))\\
 & =E_{W^{*}(Y_{j}^{t}:j\in I)}J(\sqrt{t}\bar{C}_{j}^{(h)}:(Y_{j}+\sqrt{t}\sum_{h'\neq h}\bar{C}_{j}^{(h')})\cup(Y_{i}^{t}:i\in I\setminus\{j\}))\\
 & =|G|^{1/2}t^{-1/2}E_{W^{*}(Y_{j}^{t}:j\in I)}((\bar{C}_{j}^{(h)})^{*}),
\end{align*}
the last equality by freeness. This gives (\ref{eq:expectationCh}).
On the other hand,
\[
\xi_{j}^{t}=t^{-1/2}E_{W^{*}(Y_{j}^{t}:j\in I)}(C_{j}^{(e)}).
\]
It follows that
\begin{align*}
\langle\xi_{j}^{t},(C_{j}^{(e)})^{*}\rangle & =\langle\xi_{j}^{t},E_{W^{*}(Y_{j}^{t}:j\in I)}((C_{j}^{(e)})^{*})\rangle\\
 & =t^{1/2}\langle\xi_{j}^{t},\xi_{j}^{t}\rangle\\
 & =|G|^{-1/2}\langle\xi_{j}^{t},E_{W^{*}(Y_{j}^{t}:j\in I)}((\bar{C}_{j}^{(h)})^{*})\rangle\\
 & =|G|^{-1/2}\langle\xi_{j}^{t},(\bar{C}_{j}^{(h)})^{*}\rangle,
\end{align*}
which readily implies the remaining statements of the Lemma.
\end{proof}
\begin{thm}
\label{thm:nonmicroInequality}Let $X$ be an arbitrary generating
set for $M^{G}$, and let $X\cup(\hat{u}_{g})_{g\in\hat{G}}$ be the
generating set for $M$ and $X\cup(\hat{u}_{g})_{g\in\hat{G}}\cup(u_{g})_{g\in G}$
be the generating set for $M\rtimes_{\alpha}G$ as constructed in
$\S$\ref{subsec:SpecialGenerators}. Then for any $\epsilon>0$ there
exists a $\lambda>0$ so that
\begin{equation}
\delta^{\star}((\lambda X_{i}:i\in I)\cup(\lambda\hat{u}_{h}:h\in\hat{G})\cup(u_{g}:g\in G)-1\leq|G|^{-1}(\delta^{\star}(Y_{i}:i\in I)-1)+\epsilon.\label{eq:mainInequality}
\end{equation}
\end{thm}

\begin{proof}
Let $I=\{1,\dots,d\}\sqcup\hat{G}$; for $i\in I$ set $Y_{i}=X_{i}$
if $i\in\{1,\dots,d\}$ and $Y_{i}=u_{h}$ if $i=h\in\hat{G}$. Let
also $\omega_{i}:G\to\mathbb{C}$ be given by $\omega_{i}(g)=1$ if
$i\in\{1,\dots,d\}$ and $\omega_{i}(g)=\langle g,h\rangle$ if $i=h\in\hat{G}$.
We then have the following relations:
\begin{align*}
u_{g}Y_{i}u_{g}^{*} & =\omega_{i}(g)Y_{i},\qquad i\in I,\ g\in G.
\end{align*}
For $i\in I$ let $C_{i}$ be a circular system, free from $W^{*}(Y_{i}:i\in I,u_{g}:g\in G)$,
and for $g\in G$ let $C'_{g}$ be another circular system free from
$W^{*}(Y_{i},C_{i}:i\in I,u_{g}:g\in G)$. We then have:
\begin{align*}
\delta^{\star}(X\cup(\hat{u}_{h})_{h\in\hat{G}}\cup(u_{g})_{g\in G})=2(d+|\hat{G}| & +|G|)\\
-\limsup_{t\to0}t\Bigg[\sum_{j\in I}\Vert J(Y_{j}+\sqrt{t}C_{j}) & :(Y_{i}:i\in I\setminus\{j\})\cup(u_{g}+\sqrt{t}C_{g}':g\in G\})\Vert_{2}^{2}\\
+\sum_{g\in G}\Vert J(u_{g}+\sqrt{t}C'_{g} & :(Y_{i}:i\in I\setminus\{j\})\cup(u_{g'}+\sqrt{t}C_{g'}:g'\in G\setminus\{g\})\Vert_{2}^{2}\Bigg].
\end{align*}

Denote by $M_{t}$ the von Neumann algebra $W^{*}(Y_{j}+\sqrt{t}C_{j}:j\in I)$
and by $\hat{M}_{t}$ the von Neumann algebra $W^{*}(M_{t},u_{g}:g\in G)$.

Using \cite{dvv:entropy5} we note that
\begin{align*}
\Vert J(u_{g}+\sqrt{t}C'_{g} & :(Y_{i}:i\in I\setminus\{j\})\cup(u_{g'}+\sqrt{t}C_{g'}:g'\in G\setminus\{g\})\Vert_{2}^{2}\\
 & \geq\Vert J(u_{g}+\sqrt{t}C'_{g}:(u_{g}'+\sqrt{t}C_{g'}:g'\in G\setminus\{g\})\Vert_{2}^{2}
\end{align*}
and by \cite{shlyakht-mineyev:freedim} we get that
\begin{align*}
2|G|-\limsup_{t\to0}t\Vert J(u_{g}+\sqrt{t}C'_{g} & :(u_{g}'+\sqrt{t}C_{g'}:g'\in G\setminus\{g\})\Vert_{2}^{2}\\
 & =\delta^{*}(u_{g}:g\in G)=\beta_{1}^{(2)}(G)-\beta_{0}^{(2)}(G)+1=1-|G|^{-1}.
\end{align*}

Let $\zeta_{j}^{t}=J(Y_{j}+\sqrt{t}C_{j}:(Y_{i}:i\in I\setminus\{j\})\cup(u_{g}+\sqrt{t}C_{g}':g\in G\})$
and set $\xi_{j}^{t}=J(Y_{j}+\sqrt{t}C_{j}:(Y_{i}:i\in I\setminus\{j\}))$.
For $h\in\hat{G}$, denote by $\bar{C}_{j}^{(h)}$ the projection
of $C_{j}$ onto the subspace of $\operatorname{span}\{u_{g}C_{j}u_{g}^{*}:g\in G\}$
consisting of vectors $x$ so that $\{u_{g}(x)u_{g}^{*}=\langle g,h\rangle x\}$.
Let also
\[
\bar{\zeta}_{j}^{(h)}=t^{-1/2}E_{\hat{M}_{t}}((\bar{C}_{j}^{(h)})^{*}).
\]
Then by applying Lemma \ref{lem:compenentsEqual} with $C_{j}^{(e)}=C_{j}$,
$C_{j}^{(g)}=u_{g}C_{j}u_{g}^{*}$, we have
\begin{align*}
\xi_{j}^{t} & =|G|^{1/2}t^{-1/2}E_{M_{t}}((\bar{C}_{j}^{(h)})^{*})
\end{align*}
It follows that
\begin{align*}
\langle t^{-1/2}\zeta_{j}^{t},(\bar{C}_{j}^{(h)})^{*}\rangle & =\langle E_{\hat{M}_{t}}(C_{j}^{*}),(\bar{C}_{j}^{(h)})^{*}\rangle\\
 & =\langle C_{j}^{*},E_{\hat{M}_{t}}(\bar{C}_{j}^{(h)})^{*}\rangle\\
 & =\langle C_{j}^{*},E_{\hat{M}_{t}}(|G|^{-1}\sum_{g\in G}\langle h,g\rangle u_{g}(\bar{C}_{j}^{(h)})^{*}u_{g}^{*})\rangle\\
 & =\langle C_{j}^{*},E_{\hat{M}_{t}}(|G|^{-1}\sum_{g\in G}\langle h,g\rangle(u_{g}+\sqrt{t}C'_{g})(\bar{C}_{j}^{(h)})^{*}(u_{g}^{*}+\sqrt{t}C'_{g})\rangle+O(t^{1/2})\\
 & =\langle C_{j}^{*},|G|^{-1}\sum_{g\in G}\langle h,g\rangle(u_{g}+\sqrt{t}C'_{g})E_{\hat{M}_{t}}((\bar{C}_{j}^{(h)})^{*})(u_{g}^{*}+\sqrt{t}C'_{g})\rangle+O(t^{1/2})\\
 & =\langle C_{j}^{*},|G|^{-1}\sum_{g\in G}\langle h,g\rangle(u_{g})E_{\hat{M}_{t}}((\bar{C}_{j}^{(h)})^{*})(u_{g}^{*})\rangle+O(t^{1/2})\\
 & =\langle|G|^{-1}\sum_{g\in G}\langle g,h\rangle u_{g^{-1}}C_{j}^{*}u_{g},E_{\hat{M_{t}}}((\bar{C}_{j}^{(h)})^{*}\rangle+O(t^{1/2})\\
 & =\langle(\bar{C}_{j}^{(h)})^{*},E_{\hat{M_{t}}}((\bar{C}_{j}^{(h)})^{*}\rangle+O(t^{1/2})\\
 & =\Vert E_{\hat{M_{t}}}((\bar{C}_{j}^{(h)})^{*})\Vert_{2}^{2}+O(t^{1/2})\\
 & \geq\Vert E_{M_{t}}((\bar{C}_{j}^{(h)})^{*})\Vert_{2}^{2}+O(t^{1/2})\\
 & =t|G|^{-1}\Vert\xi_{j}^{t}\Vert_{2}^{2}+O(t^{1/2}).
\end{align*}

We now claim that for all by one $h\in\hat{G}$, $\langle t^{1/2}\zeta_{j}^{t},(\bar{C}_{j}^{(h)})^{*}\rangle$
is almost $2|G|^{-1}$. 

Let us use the notation $x(t)\in\hat{M}_{t}+O(t^{\gamma})$ to signify
that there exists an element $y(t)\in\hat{M}_{t}$ so that $\Vert x(t)-y(t)\Vert_{2}=O(t^{\gamma})$.
Then, since $(u_{g}+\sqrt{t}C'_{g})(Y_{j}+\sqrt{t}C_{j})(u_{g}+\sqrt{t}C'_{g})^{*}\in\hat{M}_{t}$
and
\begin{align*}
(u_{g}+\sqrt{t}C'_{g})(Y_{j}+\sqrt{t}C_{j})(u_{g}+\sqrt{t}C'_{g})^{*} & =u_{g}Y_{j}u_{g}^{*}+\sqrt{t}(C_{g}'Y_{j}u_{g}^{*}+u_{g}Y_{j}(C_{g}')^{*}+u_{g}C_{j}u_{g}^{*})+O(t)\\
 & =\omega_{j}(g)Y_{j}+\sqrt{t}(C_{g}'Y_{j}u_{g}^{*}+u_{g}Y_{j}(C_{g}')^{*}+u_{g}C_{j}u_{g}^{*})+O(t),
\end{align*}
we deduce that
\[
(u_{g}+\sqrt{t}C'_{g})(Y_{j}+\sqrt{t}C_{j})(u_{g}+\sqrt{t}C'_{g})^{*}-\omega_{j}(g)(Y_{j}+\sqrt{t}C_{j})\in\hat{M}_{t}+O(t).
\]
Thus also
\[
C_{g}'Y_{j}u_{g}^{*}+u_{g}Y_{j}(C_{g}')^{*}-\omega_{j}(g)C_{j}+u_{g}C_{j}u_{g}^{*}\in\hat{M}_{t}+O(t).
\]
Hence
\[
C_{g}'Y_{j}u_{g}^{*}+u_{g}Y_{j}(C_{g}')^{*}-\omega_{j}(g)C_{j}+u_{g}C_{j}u_{g}^{*}\in\hat{M}_{t}+O(t^{1/2}).
\]

Since $u_{g}^{*}u_{g}=1$, we similarly deduce
\[
C_{g}'u_{g}^{*}+u_{g}(C_{g}')^{*}\in\hat{M}_{t}+O(t^{1/2}),
\]
so that, noting that $u_{g}\in\hat{M}_{t}+O(t^{1/2})$
\[
(C_{g}')^{*}+u_{g}^{*}C_{g}'u_{g}^{*}\in\hat{M}_{t}+O(t^{1/2}).
\]
This gives
\[
C'_{g}Y_{j}u_{g}^{*}-u_{g}Y_{j}u_{g}^{*}C_{g}'u_{g}^{*}-\omega_{j}(g)C_{j}+u_{g}C_{j}u_{g}^{*}\in\hat{M}_{t}+O(t^{1/2}).
\]

Projecting onto eigenspaces for the $G$ action (noting again that
$u_{g}\in\hat{M}_{t}+O(t^{1/2})$) gives us
\[
(\bar{C}'){}_{g}^{(h\cdot\omega_{j})}Y_{j}u_{g}^{*}-u_{g}Y_{j}u_{g}^{*}\bar{C}_{g}'^{(h\cdot\omega_{j})}u_{g}^{*}-\omega_{j}(g)\bar{C}_{j}^{(h)}+u_{g}\bar{C}_{j}^{(h)}u_{g}^{*}\in M_{t}+O(t^{1/2}).
\]
We note also that $(\bar{C}_{g}')^{(h\cdot\omega_{j})}u_{g}^{*}=\langle h,g\rangle\omega_{j}(g)\cdot u_{g}^{*}(\bar{C}')_{g}^{(h)}$
and $u_{g}C_{j}^{(h)}u_{g}^{*}=\langle h,g\rangle C_{j}^{(h)}$ whence
\[
\bar{(C}')_{g}^{(h\cdot\omega_{j})}Y_{j}u_{g}^{*}-\langle h,g\rangle\omega_{j}(g)u_{g}Y_{j}u_{g}^{*}u_{g}^{*}\bar{C}_{g}'^{(h\cdot\omega_{j})}+(\langle h,g\rangle-\omega_{j}(g))\bar{C}_{j}^{(h)}\in M_{t}+O(t^{1/2}).
\]
It follows that
\begin{align*}
\langle t^{1/2}\zeta_{j}^{t} & ,(\bar{C}_{j}^{(h)})^{*}\rangle=\langle E_{\hat{M}_{t}}(C_{j}),\bar{C}_{j}^{(h)}\rangle\\
 & =\langle C_{j},E_{\hat{M}_{t}}(\bar{C}_{j}^{(h)})\rangle\\
 & =\sum_{h'\in\hat{G}}\langle\bar{C}_{j}^{(h')},E_{\hat{M}_{t}}(\bar{C}_{j}^{(h)})\rangle\\
 & =\langle\bar{C}_{j}^{(h)},E_{\hat{M}_{t}}(\bar{C}_{j}^{(h)})\rangle+O(t^{1/2})\\
 & =\Vert E_{\hat{M}_{t}}(\bar{C}_{j}^{(h)})\Vert_{2}^{2}+O(t^{1/2})\\
 & \geq\max_{g}\Bigg|\Bigg\langle\bar{C}_{j}^{(h)},\frac{(\bar{C}')_{g}^{(h\cdot\omega_{j})}Y_{j}u_{g}^{*}-\langle h,g\rangle\omega_{j}(g)u_{g}Y_{j}u_{g}^{*}u_{g}^{*}\bar{C}_{g}'^{(h\cdot\omega_{j})}+(\langle h,g\rangle-\omega_{j}(g))\bar{C}_{j}^{(h)}}{\Vert(\bar{C}')_{g}^{(h\cdot\omega_{j})}Y_{j}u_{g}^{*}-\langle h,g\rangle\omega_{j}(g)u_{g}Y_{j}u_{g}^{*}u_{g}^{*}\bar{C}_{g}'^{(h\cdot\omega_{j})}+(\langle h,g\rangle-\omega_{j}(g))\bar{C}_{j}^{(h)}\Vert}\Bigg\rangle\Bigg|^{2}\\
 & \qquad\qquad+O(t^{1/2})\\
 & =2|G|^{-1}\max_{g}\frac{|\omega_{j}(g)-\langle h,g\rangle|^{2}}{|\omega_{j}(g)-\langle h,g\rangle|^{2}+4\Vert Y_{j}\Vert^{2}}+O(t^{1/2}).
\end{align*}
Thus if $h\neq\omega_{j}$, we have
\[
\langle t^{1/2}\zeta_{j}^{t},(\bar{C}_{j}^{(h)})^{*}\rangle\geq\frac{2|G|^{-1}}{1+\kappa}+O(t^{2}),
\]
where
\[
\kappa=\sup_{j,h}\min_{g}\frac{4\Vert Y_{j}\Vert_{2}^{2}}{|\omega_{j}(g)-\langle h,g\rangle|^{2}}=C\sup_{j}\Vert Y_{j}\Vert_{2}^{2},
\]
where $C$ is some constant that only depends on $G$. 

We can now compute, using that $\zeta_{j}^{t}=t^{-1/2}E_{\hat{M}_{t}}(C_{j}^{*})$,
\begin{align*}
\limsup_{t\to0}t\sum_{j\in I}\Vert J(Y_{j}+\sqrt{t}C_{j}) & :(Y_{i}:i\in I\setminus\{j\})\cup(u_{g}+\sqrt{t}C_{g}':g\in G\})\Vert_{2}^{2}\\
= & \limsup_{t\to\infty}\langle t^{1/2}\zeta_{j}^{t},t^{1/2}\zeta_{j}^{t}\rangle=\limsup_{t\to\infty}\langle t^{1/2}\zeta_{j}^{t},C_{j}^{*}\rangle\\
= & \limsup_{t\to0}\sum_{h\in\hat{G}}\sum_{j\in I}\langle t^{1/2}\zeta_{j}^{t},(\bar{C}_{j}^{(h)})^{*}\rangle+O(t^{1/2})\\
\geq & 2|I|(|G|-1)|G|^{-1}(1+\kappa)^{-1}+\limsup_{t\to0}|G|^{-1}t\Vert\xi_{j}^{t}\Vert_{2}^{2}+O(t^{1/2})\\
\geq & 2|I|(|G|-1)|G|^{-1}(1+\kappa)^{-1}+2|I||G|^{-1}-|G|^{-1}\delta^{\star}(Y_{i}:i\in I)\\
= & 2|I|(1+\kappa)^{-1}-2|I||G|^{-1}(1+\kappa)^{-1}+2|I||G|^{-1}-|G|^{-1}\delta^{\star}(Y_{i}:i\in I)\\
= & 2|I|(1+\kappa)^{-1}-|G|^{-1}\delta^{\star}(Y_{i}:i\in I)+2I|G|^{-1}(1-(1+\kappa)^{-1}).
\end{align*}
Putting all this together gives us
\begin{align*}
\delta^{\star}((Y_{i} & :i\in I)\cup(u_{g}:g\in G)=2|I|+2|G|\\
 & \qquad-\limsup_{t\to0}\Bigg[\sum_{j\in I}\Vert J(Y_{j}+\sqrt{t}C_{j}):(Y_{i}:i\in I\setminus\{j\})\cup(u_{g}+\sqrt{t}C_{g}':g\in G\})\Vert_{2}^{2}\\
 & \qquad\qquad+\sum_{g\in G}\Vert J(u_{g}+\sqrt{t}C'_{g}:(Y_{i}:i\in I\setminus\{j\})\cup(u_{g'}+\sqrt{t}C_{g'}:g'\in G\setminus\{g\})\Vert_{2}^{2}\Bigg]\\
\leq & 2|I|+2|G|-2|I|(1+\kappa)^{-1}+|G|^{-1}\delta^{\star}(Y_{i}:i\in I)-2I|G|^{-1}(1-(1+\kappa)^{-1})-2|G|+1-|G|^{-1}\\
= & |G|^{-1}\delta^{\star}(Y_{i}:i\in I)+1-|G|^{-1}+2|I|(1-|G|^{-1})(1-(1+\kappa)^{-1}).
\end{align*}
Thus
\begin{align*}
\delta^{\star}((Y_{i} & :i\in I)\cup(u_{g}:g\in G)-1\leq|G|^{-1}\delta^{\star}(Y_{i}:i\in I)-|G|^{-1}+2|I|(1-|G|^{-1})(1-(1+\kappa)^{-1})\\
 & =|G|^{-1}(\delta^{\star}(Y_{i}:i\in I)-1)+2|I|(1-|G|^{-1})(1-(1+\kappa)^{-1}).
\end{align*}

Suppose now we rescale $Y_{i}$ by replacing $Y_{i}$ with $\lambda Y_{i}$.
Then
\[
\delta^{\star}(Y_{i}:i\in I)=\delta^{\star}(\lambda Y_{i}:i\in I)
\]
 so that we get 
\begin{align*}
\delta^{\star}((\lambda Y_{i} & :i\in I)\cup(u_{g}:g\in G)-1\\
 & \leq|G|^{-1}(\delta^{\star}(Y_{i}:i\in I)-1)+2|I|(1-|G|^{-1})(1-(1+\lambda^{2}\kappa)^{-1}).
\end{align*}
Thus choosing $\lambda$ small enough we can ensure that
\[
\delta^{\star}((\lambda Y_{i}:i\in I)\cup(u_{g}:g\in G)-1\leq|G|^{-1}(\delta^{\star}(Y_{i}:i\in I)-1)+\epsilon
\]
as claimed.
\end{proof}
Let $Z_{1},\dots,Z_{n}$ be generators of a tracial von Neumann algebra,
it would be natural to expect that $\delta^{\star}(Z_{1},\dots,Z_{n})$
is an algebraic invariant: if $Z_{1}',\dots,Z_{n'}'$ is another set
of generators for the (non-closed) algebra $*\operatorname{-alg}(Z_{1},\dots,Z_{n})$,
then $\delta^{\star}(Z_{1},\dots,Z_{n})=\delta^{\star}(Z'_{1},\dots,Z'_{n'})$.
In particular, one expects that for any nonzero numbers $\lambda_{1},\dots,\lambda_{n}$,
\begin{equation}
\delta^{\star}(Z_{1},\dots,Z_{n})=\delta^{\star}(\lambda_{1}Z_{1},\dots,\lambda_{n}Z_{n}).\label{eq:lambdaInvariance}
\end{equation}
If this were true, then we could combine the inequality in Theorem
\ref{thm:nonmicroInequality} with the equality $\delta^{\star}((\lambda Y_{i}:i\in I)\cup(u_{g}:g\in G))=\delta^{\star}((Y_{i}:i\in I)\cup(u_{g}:g\in G))$
to deduce that $\delta^{\star}((Y_{i}:i\in I)\cup(u_{g}:g\in G))-1\leq|G|^{-1}(\delta^{\star}(Y_{i}:i\in I)-1)+\epsilon$
for all $\epsilon>0$, and conclude that $\delta^{\star}((Y_{i}:i\in I)\cup(u_{g}:g\in G))-1\leq|G|^{-1}(\delta^{\star}(Y_{i}:i\in I)-1)$.
However, to our embarrassment, we could not find a proof of (\ref{eq:lambdaInvariance}).
Note, however, that when $*\operatorname{-alg}(Z_{1},\dots,Z_{n})$
is isomorphic to a group algebra, then algebraic invariance holds
\cite{shlyakht-mineyev:freedim}.

The difficulty in the proof of Theorem \ref{thm:nonmicroInequality}
arises from the complicated form that the relation $u_{g}^{*}Y_{i}u_{g}=\omega_{i}(g)Y_{i}$
takes when we substitute $u_{g}+\sqrt{t}C'_{g}$ for $u_{g}$ and
$Y_{i}+\sqrt{t}C_{i}$ for $Y_{i}$. If we instead redefine $\hat{M}_{t}$
as $W^{*}(M_{t},u_{g}:g\in G)$ and set $\zeta_{j}^{t}=t^{-1/2}E_{\hat{M}_{t}}(C_{j})$,
then it is easy to show that $\Vert\bar{\zeta}_{j}^{t,(h)}\Vert_{2}^{2}=2t^{-1}$
if $h\neq\omega_{i}$. Indeed, since now $u_{g}\in\hat{M}_{t}$ we
see that
\[
\hat{M}_{t}\ni u_{g}(Y_{i}+\sqrt{t}C_{g})u_{g}^{*}-\omega_{i}(g)(Y_{i}+\sqrt{t}C_{g})=\sqrt{t}(u_{g}C_{g}u_{g}^{*}-\omega_{i}(g)C_{g}),
\]
so that $u_{g}C_{g}u_{g}^{*}-\omega_{i}(g)C_{g}\in\hat{M}_{t}$. Decomposing
into orthogonal components according to $h\in\hat{G}$ then gives
that $\langle h,g\rangle-\omega_{j}(g))\bar{C}_{j}^{(h)}\in\hat{M}_{t}$
so that $\bar{C}_{j}^{(h)}\in\hat{M}_{t}$ whenever $h\neq\omega_{j}$;
thus $\bar{\zeta}_{j}^{t,h}=t^{-1/2}E_{\hat{M}_{t}}(\bar{C}_{j}^{(h)})=t^{-1/2}\bar{C}_{j}^{(h)}$
and the claimed equality on the norm follows. Using this we can prove
the following inequality:
\[
\delta^{\star}((Y_{i}:i\in I)\ |\ (u_{g}:g\in G))+(1-|G|^{-1})\leq|G|^{-1}(\delta^{\star}(Y_{i}:i\in I)-1)
\]
where we define a kind of ``relative non-microstates free entropy
dimension''  $\delta^{\star}((Y_{i}:i\in I)\ |\ (u_{g}:g\in G))=2|I|-\limsup_{t}t\sum_{i}\Vert J(Y_{k}:(Y_{j}:j\in I\setminus\{i\})\cup(u_{g}:g\in G\})\Vert_{2}^{2}$.
We summarize this as:
\begin{rem}
Let $X$ be an arbitrary generating set for $M^{G}$, and let $X\cup(\hat{u}_{g})_{g\in\hat{G}}$
be the generating set for $M$ and $X\cup(\hat{u}_{g})_{g\in\hat{G}}\cup(u_{g})_{g\in G}$
be the generating set for $M\rtimes_{\alpha}G$ as constructed in
$\S$\ref{subsec:SpecialGenerators}. Then
\[
\delta^{\star}((Y_{i}:i\in I)\ |\ (u_{g}:g\in G))+(1-|G|^{-1})\leq|G|^{-1}(\delta^{\star}(Y_{i}:i\in I)-1).
\]
\end{rem}

\section{Some Applications.}
\begin{cor}
\label{cor:InequaltitySubfactor}Suppose that $G$ is a finite abelian
group acting properly outer on a factor $M$. Suppose that $M^{G}$
is generated by $d$ elements. Then $M\rtimes_{\alpha}G$ has a generating
set $S$ satisfying $\delta_{0}(S)\leq(2d+2)|G|^{-1}+1$.
\end{cor}

\begin{proof}
If $|G|=1$, there is nothing to prove, so let us assume that $|G|\geq2$.
By Theorem \ref{thm:nonmicroInequality} we have, for every $\epsilon>0$
the existence of $\lambda>0$ so that
\[
\delta^{\star}(\lambda X\cup(\lambda\hat{u}_{g})_{g\in\hat{G}}\cup(u_{g})_{g\in G})-1\leq|G|^{-1}\cdot\left[\delta^{\star}(X\cup(\hat{u}_{h})_{h\in\hat{G}})-1\right]+\epsilon.
\]
where $X$ is any generating set for $M^{G}$. We can thus assume
that $|X|\leq d$. 
\begin{align*}
\delta^{\star}(X\cup(\hat{u}_{h})_{h\in\hat{G}}) & \leq\delta^{\star}(X)+\delta^{\star}((\hat{u}_{h})_{h\in\hat{G}})\\
 & \leq2d+1-|\hat{G}|^{-1}.
\end{align*}
since $\hat{G}$ is abelian \cite{shlyakht-mineyev:freedim}. 

Let $\epsilon=|\hat{G}|^{-1}$. Then by Theorem \ref{thm:nonmicroInequality}
there exists some $\lambda>0$ so that if we set $S=X\cup(\hat{u}_{g})_{g\in\hat{G}}\cup(u_{g})_{g\in G}$
so that (\ref{eq:mainInequality}) holds. Combining this with the
remarkable inequality between microstates and non-microstates free
entropy \cite{guionnet-biane-capitaine:largedeviations} and invariance
of $\delta_{0}$ under algebraic changes of variables \cite{dvv:improvedrandom}
we obtain:
\begin{align*}
\delta_{0}(S) & =\delta_{0}(\lambda X\cup(\lambda\hat{u}_{g})_{g\in\hat{G}}\cup(u_{g})_{g\in G}))\\
 & \leq\delta^{\star}(\lambda X\cup(\lambda\hat{u}_{g})_{g\in\hat{G}}\cup(u_{g})_{g\in G}))\\
 & \leq|G|^{-1}(2d+1-|\hat{G}|^{-1})+1+\epsilon=2d|G|^{-1}+|G|^{-1}-|G|^{-2}+1+|G|^{-1}\\
 & \leq(2d+2)|G|^{-1}+1=(2d+2)|G|^{-1}+1.
\end{align*}
as claimed.
\end{proof}
\begin{thm}
Let $M$ be a finitely generated factor, and assume that $M\cong M_{2\times2}(M)$
and that  $\alpha$ is a properly outer action of $G=(\mathbb{Z}/2\mathbb{Z})^{\oplus\infty}$
on $M$. Then for every $\epsilon>0$ there exists a finite generating
set $S$ for $M\rtimes_{\alpha}G$ so that $\delta_{0}(S)\leq1+\epsilon$. 
\end{thm}

\begin{proof}
Suppose that $M$ is generated by $d$ elements. Choose $m$ so that
$2^{-m}(2d+2m+2)<\epsilon$. 

Denote by $G_{m}$ the subgroup of $G$ generated by first $m$ copies
of $\mathbb{Z}/2\mathbb{Z}$. Then $M^{G_{m}}\cong N\rtimes G_{m}$
where $N$ is a II$_{1}$ factor so that $M_{2^{m}\times2^{m}}(N)\cong M$.
Thus by assumption $N\cong M$, so $N$ can be generated by $d$ elements.
Thus $M^{G_{m}}$ is generated a set $S'$ of at most $d+m$ elements.
Applying now Corollary \ref{cor:InequaltitySubfactor} we deduce that
there exists a generating set $S_{1}$ for $M\rtimes G_{m}$ so that
$\delta_{0}(S_{1})\leq2^{-m}(2d+2m+2)+1$; moreover, $S_{1}$ includes
the set $S_{0}$ consisting of generators of $G_{m}$. 

$G$ is an infinite abelian group whose topological dual is isomorphic
to the Cantor set. Thus the von Neumann algebra of $G$ is generated
by a single unitary $w$. Let $S_{2}$ be that unitary. Then $S=S_{0}\cup S_{1}\cup S_{2}$
is a generating set for $M\rtimes_{\alpha}G$. 

By the hyperfinite inequality for $\delta_{0}$ \cite{jung:dimHyperfinite},
we have

\begin{align*}
\delta_{0}(S) & =\delta_{0}(S_{1}\cup S_{2}\cup S_{0})\leq\delta_{0}(S_{1}\cup S_{0})+\delta_{0}(S_{2}\cup S_{0})-\delta_{0}(S_{0})\\
 & \leq2^{-m}(2d+2m+2)+1+1-\delta(S_{0}).
\end{align*}
Since $S_{0}$ generates $G_{m}$ which is a finite abelian group,
$\delta_{0}(S_{0})=1-|G_{m}|^{-1}=1-2^{-m}$ \cite{jung:dimHyperfinite,dvv:entropy3}.
Substituting this into the inequality above gives
\[
\delta_{0}(S)\leq2^{-m}(2d+2m+2)+1-(1-2^{-m})=2^{-m}(2d+2m+2)+1<1+\epsilon
\]
as claimed.
\end{proof}
\bibliographystyle{alpha}
%\bibliography{tex/quasifree}

\end{document}